\documentclass[12pt]{amsart}

\usepackage[all]{xy}
\usepackage{amsmath}
\usepackage{amssymb}
\usepackage{amsfonts}
\usepackage{amsthm}
\usepackage{verbatim}
\usepackage{graphics}
\usepackage{graphicx,color}

\newcommand{\R}{\mathbb{R}}
\newcommand{\C}{\mathbb{C}}

\newcommand{\g}{\mathfrak{g}}
\newcommand{\kk}{\mathfrak{k}}
\newcommand{\q}{\mathfrak{q}}
\newcommand{\lev}{\mathfrak{l}}
\newcommand{\p}{\mathfrak{p}}
\newcommand{\uu}{\mathfrak{u}}
\newcommand{\ta}{\mathfrak{t}}

\def\GL{{\rm GL}}
\def\SL{{\rm SL}}

\def\SO{{\rm SO}}
\def\GSp{{\rm GSp}_4}
\def\Sp{{\rm Sp}}

\def\PGL{{\rm PGL}}

\def\Ind{{\rm Ind}}

\def\<{\langle}
\def\>{\rangle}

\newcommand{\nc}{\newcommand}
\nc{\nt}{\newtheorem}
\nc{\dmo}{\DeclareMathOperator}
\nc{\enm}{\ensuremath}

\numberwithin{equation}{section}
\newtheorem{rmk}[equation]{Remark}
\newtheorem{thm}[equation]{Theorem}
\newtheorem{conjecture}[equation]{Conjecture}
\newtheorem{prop}[equation]{Proposition}
\newtheorem{lemma}[equation]{Lemma}
\nt{defn}[equation]{Definition}

\nt{assumption}[equation]{Assumption}

\title[\bf Transfer of cohomological representations]{Functorial transfer of Cohomological Representations from $\Sp(4,\R)$ to $\GL(5,\R)$}
\author{\bf Makarand Sarnobat}
\date{\today}
\address{Indian Institute of Science Education and Research, Bhopal Bypass Road, Bhauri, Bhopal 462 066, Madhya Pradesh INDIA.}
\email{makarand.sarnobat16@gmail.com}

\begin{document}

\maketitle

\begin{abstract}
Let $G=\Sp(4,\R)$ and let $\pi$ be an irreducible, unitary representation of $G$ which is cohomological with respect to trivial coefficients. Using the inclusion from $\SO(5,\C)$ to $\GL(5,\C)$, we transfer $\pi$ to an irreducible representation $\iota(\pi)$ of $\GL(5,\R)$ and determine how the property of being cohomological behaves under Langlands functoriality. We also consider representations which are cohomological with respect to non-trivial coefficients.
\end{abstract}

\section{Introduction}

The main aim of this article is to study the cohomological properties of representations of $\GL(5,\R)$ which are obtained by a Langlands transfer of cohomological representations of $\Sp(4,\R)$. This project started with the following observation of Labesse and Schwermer: A cohomological representation $\pi$ of $\GL(2,\R)$ transfers to a cohomological representation of $\GL(3,\R)$ via the symmetric power transfer (see \cite{labesse-schwermer}). This result was extended for the symmetric power transfer from $\GL(2,\R)$ to $\GL(n+1,\R)$ by Raghuram in \cite{Ra}. Such a result was then used to study the arithmetic of symmetric power $L$-functions attached to $\pi.$ The reader is also referred to \cite{Raghuram-Shahidi} where there is a general discussion involving Langlands functoriality, cohomological representations, and applications to the special values of $L$-functions. Further, in \cite{Ra-Sa}, when does a tempered representation of a classical group transfer to a cohomological representation of an appropriate $\GL(n,\R)$ or $\GL(n,\C)$ was determined by the author and Raghuram. This led to the following question: When does a cohomological representation (tempered or not) of a classical group transfer to a cohomological representation of an appropriate $\GL(n,\R)$ or $\GL(n,\C)$?

We answer this question completely in the special case of transferring cohomological representations with trivial coefficients from $\Sp(4,\R)$ to $\GL(5,\R)$. The main result of this article is Theorem \ref{sp4-main-result-triv-coeff}, which states that a cohomological representation of $\Sp(4,\R)$ is transferred to a cohomological representation if it is the trivial representation, a discrete series representation or it is induced from the Siegel parabolic.  We also work out a toy case of transfer from $\SL(2,\R)$ to $\GL(3,\R)$ in Section \ref{SL2}. This example does not shed much light on whether cohomologicalness is preserved under functoriality in general since the only cohomological representations of $\SL(2,\R)$ are the discrete series representations and the trivial representation. The two main tools which are used in proving Theorem \ref{sp4-main-result-triv-coeff} are the Vogan-Zuckerman classification of unitary, irreducible cohomological representations (which will be called cohomological representations)which can be found in \cite{Vo-Zu} and a similar classification for $\GL(5,\R)$ given by Speh in \cite{Sp2}. Section \ref{Notations} introduces basic definitions and fixes notations. Sections \ref{V-Z Classification} and \ref{Speh} recall the Vogan-Zuckerman classification of cohomological representations and Speh's classification. Then, using the classification of Vogan-Zuckerman, we list down all the cohomological representations of $\Sp(4,\R)$ in Section \ref{V-Z for Sp}. We then explicitly compute the transfer of these representations in section \ref{Transfer triv-coeff} and check, using Speh, which of the resulting representations are cohomological. Finally, we summarize our results for cohomological representations with trivial coefficients in Section \ref{Summary Triv-coeff} and we make further observations in the non-trivial coefficients case in Section \ref{Transfer non-triv coeff}. Though we do not have a complete result in the case of the non-trivial coefficients, we make a plausible conjecture in section \ref{Summary non-triv coeff}.

\bigskip
\noindent {\small {\it Acknowledgements:} 
I would like to thank Raghuram for suggesting this problem and giving his valuable inputs from time to time. I would also like to thank Dipendra Prasad and Arvind Nair for their interest in the results of this project, and for helpful tutorials on Langlands parameters.}

\bigskip
\section{Background and Notations}
\label{Notations}
Let $\Sp(2n,\R)=\{A \in \GL(2n,\R): {}^tAJA=J \},$ where $J=\begin{pmatrix}
0 & I_n \\
-I_n & 0
\end{pmatrix}.$ Let $\g_0$ be the corresponding real Lie algebra. Let $$\mathfrak{h}_0=\Bigg\{ \begin{pmatrix}
     &        &      & & x_1 &        & \\
     & 0      &      & &     & \ddots & \\
     &        &      & &     &        & x_n \\
-x_1 &        &      & &     &        & \\
     & \ddots &      & &     &    0   & \\
     &        & -x_n & &     &        &
\end{pmatrix} \Bigg| x_i \in \R \Bigg\}.$$
Let $K= \left\lbrace  \begin{pmatrix}
A & B \\
-B & A
\end{pmatrix}: \ A, B \in \GL(n,\R), A{}^tB={}^tBA, \  A{}^tA + B{}^tB = I_n \right\rbrace $ be a maximal compact subgroup of $\Sp(2n,\R)$ and $W_K$ be the Weyl group of $K$. Any element of $W_K$ acts on an element of $i\mathfrak{h}_0$ by permuting the entries $x_i$.\\

For $G = \Sp(4,\R)$, we fix an appropriate basis for the Lie algebra of $G$. We fix the following basis (see \cite{Am-Sc}):
\begin{align*}
Z=-i\begin{pmatrix}
0 & 0 & 1 & 0 \\
0 & 0 & 0 & 0 \\
-1 & 0 & 0 & 0 \\
0 & 0 & 0 & 0
\end{pmatrix}, \hskip 8mm  &  Z'= -i\begin{pmatrix}
0 & 0 & 0 & 0 \\
0 & 0 & 0 & 1 \\
0 & 0 & 0 & 0 \\
0 & -1 & 0 & 0
\end{pmatrix},\\
N_+=\frac{1}{2}\begin{pmatrix}
0 & 1 & 0 & -i \\
-1 & 0 & -i & 0 \\
0 & i & 0 & 1 \\
i & 0 & -1 & 0
\end{pmatrix}, \hskip 5mm &  N_- = \frac{1}{2}\begin{pmatrix}
0 & 1 & 0 & i \\
-1 & 0 & i & 0 \\
0 & -i & 0 & 1 \\
-i & 0 & -1 & 0
\end{pmatrix},
\end{align*}
\newpage
\begin{align*}
X_+=\frac{1}{2}\begin{pmatrix}
1 & 0 & i & 0 \\
0 & 0 & 0 & 0 \\
i & 0 & -1 & 0 \\
0 & 0 & 0 & 0
\end{pmatrix}, \hskip 5mm & X_-=\frac{1}{2}\begin{pmatrix}
1 & 0 & -i & 0 \\
0 & 0 & 0 & 0 \\
-i & 0 & -1 & 0 \\
0 & 0 & 0 & 0
\end{pmatrix}, \\
P_{1+} = \frac{1}{2}\begin{pmatrix}
0 & 1 & 0 & i \\
1 & 0 & i & 0 \\
0 & i & 0 & -1 \\
i & 0 & -1 & 0
\end{pmatrix}, \hskip 5mm & P_{1-} = \frac{1}{2}\begin{pmatrix}
0 & 1 & 0 & -i \\
1 & 0 & -i & 0 \\
0 & -i & 0 & -1 \\
-i & 0 & -1 & 0
\end{pmatrix},\\
P_{0+} = \frac{1}{2}\begin{pmatrix}
0 & 0 & 0 & 0 \\
0 & 1 & 0 & i \\
0 & 0 & 0 & 0 \\
0 & i & 0 & -1
\end{pmatrix}, \hskip 5mm & P_{0-} = \frac{1}{2}\begin{pmatrix}
0 & 0 & 0 & 0 \\
0 & 1 & 0 & -i \\
0 & 0 & 0 & 0 \\
0 & -i & 0 & -1
\end{pmatrix}.
\end{align*}

Note that, we have the Cartan decomposition for $\g = \mathfrak{sp}(4) = \kk \oplus \p$, where $\kk = \< Z,Z',N_+,N_- \>$ and $\p = \< X_+,X_-,P_{1+},P_{1-},P_{0+},P_{0-} \>.$

The table of Lie brackets for the above basis is as follows:
\vskip 3mm
\begin{center}
\begin{tabular}{|c|c|c|c|c|c|c|c|c|c|c|}
\hline
  & $Z$ & $Z'$ & $N_+$ & $N_-$ & $X_+$ & $X_-$ & $P_{1+}$ & $P_{1-}$ & $P_{0+}$ & $P_{0-}$\\
  \hline
$Z$ & 0 & 0 & $N_+$ & $-N_-$ & $2X_+$ & $-2X_-$ & $P_{1+}$ & $-P_{1-}$ & 0 & 0 \\
\hline
$Z'$ & 0 & 0 & $-N_+$ & $N_-$ & $0$ & $0$ & $P_{1+}$ & $-P_{1-}$ & $2P_{0+}$ & $-2P_{0-}$\\
\hline
$N_+$ & $-N_+$ & $N_-$ & 0 & $Z'-Z$ & $0$ & $-P_{1-}$ & $2X_+$ & $-2P_{0-}$ & $P_{1+}$ & $0$\\
\hline
$N_-$ & $N_-$ & $-N_-$ & $Z-Z'$ & 0 & $-P_{1+}$ & $0$ & $-2P_{0+}$ & $2X_{-}$ & $0$ & $P_{1-}$ \\
\hline
$X_+$ & $-2X_+$ & $0$ & $0$ & $P_{1+}$ & $0$ & $Z$ & $0$ & $N_+$ & $0$ & $0$ \\
\hline
$X_-$ & $2X_-$ & $0$ & $P_{1-}$ & $0$ & $-Z$ & $0$ & $N_-$ & $0$ & $0$ & $0$ \\
\hline
$P_{1+}$ & $-P_{1+}$ & $-P_{1+}$ & $-2X_+$ & $2P_{0+}$ & $0$ & $-N_-$ & $0$ & $Z+Z'$ & $0$ & $N_+$ \\
\hline
$P_{1-}$ & $P_{1-}$ & $P_{1-}$ & $2P_{0-}$ & $-2X_-$ & $-N_+$ & $0$ & $-Z-Z'$ & $0$ & $N_-$ & $0$ \\
\hline
$P_{0+}$ & $0$ & $-2P_{0+}$ & $-P_{1+}$ & $0$ & $0$ & $0$ & $0$ & $-N_-$ & $0$ & $Z'$ \\
\hline
$P_{0-}$ & $0$ & $2P_{0-}$ & $0$ & $-P_{1-}$ & $0$ & $0$ & $-N_+$ & $0$ & $-Z'$ & $0$ \\
\hline
\end{tabular}
\end{center}

\vskip 3mm
These will come in handy for computations later on.

\section[Cohomological representations]{Vogan-Zuckerman classification of cohomological representations}
\label{V-Z Classification}
We briefly recall the Vogan-Zuckerman classification for cohomological representations and an algorithm to compute the Langlands inducing data for these representations. For more details the reader is referred to \cite{Vo-Zu}.  Let $G$ be a connected real semi-simple Lie group with finite center. Let $\g_0$ be the Lie algebra of $G$ and $\g$ be the complexification of $\g_0$. Let $K \subseteq G$ be a maximal compact subgroup of $G$ and $\theta$ be the corresponding Cartan involution of $G$. Then, we have the Cartan decomposition $$\g = \kk \oplus \p,$$ where $\kk$ is the $+1$ eigenspace of $\theta$ and $\p$ the $-1$
eigenspace.

Harish-Chandra in $1953$ proved the following result:

\begin{thm}[Harish-Chandra, \cite{H-C1}]
Let $(\pi,V)$ be an irreducible unitary representation of $G$. Then $V_{K}^{{}^\infty}$ is irreducible as a $\g$-module and determines $\pi$ up to unitary equivalence, where $V_{K}^{{}^\infty}$ is the subspace of smooth $K$-finite vectors in $V$.
\end{thm}
The subspace $V_{K}^{{}^\infty}$ has a $(\g,K)$-module structure and the result implies that it is enough to study $(\g,K)$-modules. Vogan-Zuckermann describes those $(\g,K)$-modules for which the $(\g,\kk)$-cohomology groups do not vanish. We need two parameters: a $\theta$-stable parabolic subalgebra $\q$ of $\g$ and an admissible homomorphism $\lambda$ on the Levi part of $\q$.

We construct a $\theta$-stable parabolic subalgebra as follows: Let $x \in i\kk_0$. Since $K$ is compact, $ad(x): \g \rightarrow \g$ is diagonalizable with real eigenvalues. Define,
\begin{eqnarray*}
\q & = & \text{ sum of non-negative eigen-spaces of $ad(x)$},\\
\lev & = & \text{ the zero eigen-space of $ad(x)$ = centralizer of $x$},\\
\uu & = & \text{ sum of positive eigen-spaces.}
\end{eqnarray*}

Then $\q$ is a parabolic subalgebra of $\g$ and $\q = \lev + \uu$ is the Levi decomposition of $\q$. Further $\lev_0 = \g_0 \cap \lev$. Since $\theta(x) =x$, the subalgebras $\q,\lev,\uu$ are all invariant under the Cartan involution $\theta$. The subalgebra $\q$ is called a $\theta$-stable parabolic subalgebra which is one of the two parameters. Let $\ta_0 \subseteq \kk_0$ be a Cartan subalgebra containing $ix.$ Then $\ta \subseteq \lev$. For any subspace $\mathfrak{f}$, which is stable under $ad(\ta)$, let $\Delta(\mathfrak{f},\ta)=\{\alpha_1,\dots,\alpha_r\}$ be the roots of $\ta$ occurring in $\mathfrak{f}$. We allow multiplicities in the set $\Delta(\mathfrak{f},\ta)$.
Define
$$\rho(\mathfrak{f})= \frac{1}{2} \sum\limits_{\alpha_i \in \Delta(\mathfrak{f})} \alpha_i.$$
Let $L \subseteq G$ be the connected subgroup of $G$ with Lie algebra $\lev_0$. A representation $\lambda: \lev \rightarrow \C$ is called admissible if
\begin{itemize}
\item $\lambda$ is a differential of a unitary character (also denoted by $\lambda$) of $L$.
\item If $\alpha \in \Delta(\uu)$, then $\langle \alpha, \lambda\vert_{\ta} \rangle \geq 0.$
\end{itemize}
Given a $\theta$-stable parabolic subalgebra $\q$ and an admissible $\lambda$, define 
$$\mu(\q,\lambda)= \text{Representation of $K$ of highest weight } \lambda\vert_{\ta}+2\rho(\uu \cap \p).$$
We have the following classification result:

\begin{thm}[\cite{Vo-Zu} Theorem 5.3]
Let $\q$ be a $\theta$-stable parabolic subalgebra and let $\lambda: \lev \rightarrow \C$ be an admissible character. Then there is a unique irreducible $\g$-module $A_{\q}(\lambda)$ such that:
\begin{enumerate}
\item The restriction of $A_{\q}(\lambda)$ to $\kk$ contains $\mu(\q,\lambda).$
\item The center $Z(\g)$ of the universal enveloping algebra acts by $\chi_{\lambda+\rho}$ on $A_{\q}(\lambda)$.
\item If a representation of highest weight $\delta$ of $\kk$ appears in the restriction of $A_{\q}(\lambda)$, then 
$$\delta = \lambda\vert_\ta + 2\rho(\uu \cap \p) + \sum\limits_{\alpha \in \Delta(\uu \cap \p)} n_{\alpha}\alpha,$$ with $n_\alpha$'s non-negative integers.
\end{enumerate}
\end{thm}

This classifies all irreducible unitary cohomological representations of the Lie group $G$. The representation $A_\q(\lambda)$ has non-trivial cohomology with respect to the finite-dimensional representation of $G$ with highest weight $\lambda$.

\begin{rmk} \cite{Vo-Zu} \label{DS}
$A_q(\lambda)$ is a discrete series representation if and only if $\lev \subseteq \kk$. Further, $A_\q$ is a tempered representation if and only if $[\lev,\lev] \subseteq \kk$.
\end{rmk}

\medskip
\subsection{Langlands data for $A_{\q}(\lambda)$}
We obtain the Langlands inducing data for $A_{\q}(\lambda)$'s as follows. For details, the reader is referred to \cite{Vo-Zu}. Fix a maximally split $\theta$-stable Cartan subgroup $H=TA$ of $L$ (corresponding to the Levi part $\lev$ of the $\theta$-stable parabolic subalgebra $\q$) and an Iwasawa decomposition $L=(L\cap K)AN^L.$ Put
\begin{eqnarray*}
MA & = & \text{Langlands decomposition of centralizer of } A \text{ in } G,\\
\nu & = & (\frac{1}{2} \text{ sum of roots of $\mathfrak{a}$ in $\mathfrak{n}^L$})+\lambda\vert_{\mathfrak{a}} \hskip 3mm
 \in \mathfrak{a}^*.
\end{eqnarray*}

Now, let $P$ be any parabolic subgroup of $G$ with Levi factor $MA$ satisfying $\langle Re(\nu),\alpha \rangle \geq 0$ for all roots $\alpha$ of $\mathfrak{a}$ in $\mathfrak{n}^L$. The Harish-Chandra parameter of the discrete series representation of $M$ is given by $\rho^+ + \lambda\vert_{\ta}+\rho(\uu)$; where $\rho^+$ is half sum of positive roots of $\ta$ in $\mathfrak{m} \cap \lev$ and $\rho(\uu)$ is half sum of roots of $\ta$ in $\uu$. We denote this discrete series representation by $\sigma$. The only difficulty here is that if $M$ is not connected then the Harish-Chandra parameter does not completely determine the discrete series representation, $\sigma$, of $M$. We fix this as follows:\\
Let
\begin{eqnarray*}
\mu^{M}(\q,\lambda) & = & \text{Representation of $M \cap K$ of extremal weight } \\
& & \lambda \vert _{\ta} + 2\rho(\wedge^{dim(\uu \cap \p)}(\uu \cap \p))\vert _{\ta}.
\end{eqnarray*}
Let $\sigma$ be the discrete series representation with lowest $M \cap K$ type $\mu^{M}(\q,\lambda)$. This completely determines the discrete series representation of $M$. The parabolic subgroup $P$, the character,$\nu$, of $\mathfrak{a}$ and $\sigma$ the representation of $M$ gives us the Langlands inducing data for $A_{\q}(\lambda)$.

\bigskip
\section{Speh's Classification}
\label{Speh}
We now recall Speh's classification of irreducible, unitary representations of $\GL(n,\R)$ which are cohomological with respect to trivial coefficients. Let $G=\GL(2n,\R), n>1$. Let $C_n=T_nA_n$ be the Cartan subgroup containing matrices of the form:\\
$$\begin{pmatrix}
cos \phi_1 & sin \phi_1 & & &  \\
-sin \phi_1 & cos \phi_1 & & &  \\
 & & \ddots & &  \\
 & &          & cos \phi_n & sin \phi_n\\
 & &          & -sin \phi_n & cos \phi_n
\end{pmatrix}
\begin{pmatrix}
a_1 & & & &  \\
& a_1 & & &  \\
 & & \ddots & &  \\
 & &          & a_n & \\
 & &          & & a_n
\end{pmatrix}.$$

Then the roots system $\Phi(\mathfrak{c_n},\g)$ is of type $A_{n-1}$ with each root occurring $4$ times. Let $\Phi^+$ be the set of positive roots and set $\rho_n=\sum\limits_{\alpha \in \Phi^+}2\alpha$.\newline Let $P=M_nA_nN$ be the parabolic subgroup determined by the set of positive roots $\Phi^+$. The connected component $M_n^\circ$ of $M_n$ is isomorphic to $n$ copies of $SL(2,\R)$ and $T_n$ is isomorphic to a product of $n$ copies of $O(2)$.

Let $\chi(k); k > 0$ be the quasi-character of $C_n$ such that the restriction of $\chi$, to each $\SO(2)$ component, is $e^{2\pi ik}$ and the restriction to $A_n$ is $exp(\frac{1}{2}\rho_n)$. Define $$I(k)=J(\chi(k)),$$ where $J(\chi(k))$ is the Langlands quotient of the induced representation  $\Ind_{P}^{\GL(2n)}(\pi(k) \otimes \chi(k))$ and $\pi(k) = D_k \otimes D_k \otimes \dots \otimes D_k$ is a representation of $M = \SL_{\pm}(2,\R)^n$. If $G=\GL(2,\R)$, then put $I(k)$ to be the discrete series representation $D_k$  of $\GL(2,\R)$.

The cohomological representations of $\GL(n,\R)$ are obtained as follows: Let $(n_0,n_1,\dots,n_r)$ be a partition of $n$ with $n_0 \geq 0$ and $n_i = 2 m_i$ for all $1 \leq i \leq r$ and all the $n_i$ are positive. Let $P = MAN$ be the parabolic corresponding to the partition $(n_0,\dots,n_r)$. Then, $$M = \prod\limits_{i=0}^{r} \SL_{\pm}(n_i,\R).$$
Let $k_i = n - \sum\limits_{j=i+1}^{r}n_j - m_i$. Define the induced representation
$$\Ind_P^G( \otimes_{i=1}^r I(k_i) \otimes \chi_{n_0} \otimes \chi_0),$$
where $I(k_i)$ are representations of $\SL_{\pm}(n_i,\R)$ and $\chi_{n_0}$ and $\chi_0$ are trivial representations of $\SL_{\pm}(n_0,\R)$ and $AN$ respectively. Then we have, 

\begin{thm}[see \cite{Sp2}]
The induced representation $$\Ind_P^G( \otimes_{i=1}^r I(k_i) \otimes \chi_{n_0} \otimes \chi_0)$$ is irreducible and classifies all the unitary, irreducible representations of $\GL(n,\R)$ which have cohomology with trivial coefficients.
\end{thm}

For the purposes of our computations, it will be convenient for us to write down the Langlands inducing data for these representations.

With notations as above, choose a Cartan subgroup $C_{(n-n_0)/2}$ in $MA$ with the following properties:
\begin{itemize}
\item $C_{(n-n_0)/2} \cap \SL_{\pm}(n_l,\R)$ is the fundamental Cartan subgroup of $\SL_{\pm}(n_l,\R)$ for $l \geq 1$, and 
\item $C_{(n-n_0)/2} \cap \SL_{\pm}(n_0,\R)$ is the split Cartan subgroup of $\SL_{\pm}(n_0,\R)$.
\end{itemize}

Then we can decompose $C_{(n-n_0)/2}$ as $T_{(n-n_0)/2}A_{(n-n_0)/2}$ with the following properties:
\begin{itemize}
\item $T_{(n-n_0)/2} = \prod\limits_{l=0}^{r}T_{n_l/2}$ with $T_{n_l/2} = T_{(n-n_0)/2} \cap \SL_{\pm}(n_l,\R)$ for $l \geq 0$, and 
\item $A_{(n-n_0)/2} = A\prod\limits_{l=0}^{r}A_{n_l}$ with $A_{n_l} = A_{(n-n_0)/2} \cap \SL_{\pm}(n_l,\R)$.
\end{itemize}

Choose a cuspidal parabolic subgroup $Q$ containing $C_{(n-n_0)/2}$ and the upper triangular matrices and write, for $0 \leq l \leq r$, $2\rho_l$ for the sum of positive roots of $(\mathfrak{sl}(n_l,\R),\mathfrak{a}_l)$ for the sum of positive roots determined by $Q$. Let $\chi(n) \in \hat{C}_{(n-n_0)/2}$ be such that the following holds:
\begin{itemize}
\item $\chi(n)|_{A} = \chi_0$,
\item $\chi(n)|_{A_0} = \rho_0$
\item $\chi(n)|_{T_0}$ is trivial
\item $\chi(n)|_{A_l} = \frac{1}{2}\rho_l$ for $l \rangle $,
\item $\chi(n)|_{T_l}$ is a product of factors $exp((n-\sum\limits_{i=l+1}n_i - m_l)2\pi i)$.
\end{itemize}

Then, by \cite{Sp2} Proposition 4.1.1,  $\Ind_P^G( \otimes_{i=1}^r I(k_i) \otimes \chi_{n_0} \otimes \chi_0) \cong J(\chi(n))$.

\bigskip
\section{$\SL(2,\R)$ to $\GL(3,\R)$}
\label{SL2}
In this section, as a warm-up example we will study the cohomological properties of representations of $\GL(3,\R)$ which are obtained by transferring $A_{\q}(\lambda)$'s of $\SL(2,\R)$. We denote by $\mathfrak{sl}(2,\C)$ the complexified Lie algebra of $\SL(2,\R)$.
Let $H = \left( \begin{smallmatrix}
1 & 0 \\
0 & -1
\end{smallmatrix} \right), X = \left( \begin{smallmatrix}
0 & 1 \\
0 & 0 
\end{smallmatrix} \right)$ and
$Y = \left( \begin{smallmatrix}
0 & 0 \\
1 & 0 
\end{smallmatrix} \right)$ be a basis of $\mathfrak{sl}(2,\R)$.

Let $w = \left( \begin{smallmatrix}
0 & 1 \\
1 & 0
\end{smallmatrix} \right)$, $\phi(A) = -{}^{t}A$ and $\theta = int(w) \circ \phi$, where $int(w)$ is the inner automorphism by $w$. Then $\theta$ is a Cartan involution on $\mathfrak{sl}(2,\C)$ such that
$\mathfrak{sl}(2,\C) = \langle H \rangle \oplus \langle X, Y \rangle,$ is the Cartan decomposition of $\mathfrak{sl}(2,\C)$ with $\mathfrak{k} = \langle H \rangle $ and $\mathfrak{p} = \langle X, Y \rangle$.

There are three $\theta$-stable parabolic subalgebras of $\mathfrak{sl}(2,\C)$ corresponding to $0$, $H$ and $-H$.
\begin{enumerate}
\item  Corresponding to $0$: This gives the full algebra of $q_0 = \mathfrak{sl}(2,\C) = \lev$.

\item Corresponding to $H$: The parabolic subalgebra is 
$$\mathfrak{q}_1 = \langle H \rangle \oplus \langle X \rangle,$$ where $\lev = \langle H \rangle$ and $\mathfrak{u} = \langle X \rangle. $
\item Corresponding to $-H$: The parabolic subalgebra is 
$$\mathfrak{q}_2 = \langle H \rangle \oplus \langle Y \rangle,$$ where $\lev = \langle H \rangle$ and $\mathfrak{u} = \langle Y \rangle.$
\end{enumerate}

Note that the only possible admissible character $\lambda$ for $\q_0$ is $\lambda = 0.$ This gives rise to the trivial representation of $\SL(2,\R)$. This representation is transferred to the trivial representation of $\GL(3,\R)$ which is cohomological with respect to the trivial coefficients.
Observe that the Levi parts of both $\q_1$ and $\q_2$ are contained in $\mathfrak{k}$. Thus the cohomological representations $A_{\mathfrak{q}_1}(\lambda)$ and $A_{\mathfrak{q}_2}(\lambda)$ are discrete series representations with highest weight $\lambda$. The Langlands parameter for a representation of $\SL(2,\R)$ is a homomorphism from the Weil group of $\R$ to $\PGL(2,\C)$. The parameter $\phi(D_n)$ for the discrete series representation of $\SL(2,\R)$ is given by
$$z \mapsto \begin{pmatrix}
(\frac{z}{\bar{z}})^\frac{n}{2} & 0 \\
0 & 1
\end{pmatrix}, \hskip 5mm j \mapsto \begin{pmatrix}
0 & 1 \\
1 & 0
\end{pmatrix}.$$

To compute the transfer of the discrete series representations of $\SL(2,\R)$ to $\GL(3,\R)$, we embed $\PGL(2,\C)$ into $\GL(3,\R)$ via the $3$- dimensional representation induced by $\GL(2,\C)$ taking 
$$\begin{pmatrix}
a & 0 \\
0 & b
\end{pmatrix} \mapsto \begin{pmatrix}
\frac{a}{b} & 0 & 0 \\
0 & 1 & 0 \\
0 & 0 & \frac{b}{a}
\end{pmatrix}.$$ 
The image of $\PGL(2,\C)$ can be identified with $\SO(3) \subset \GL(3,\C)$ which preserves the quadratic form $\begin{pmatrix}
0 & 0 & 1 \\
0 & -1 & 0 \\
1 & 0 & 0
\end{pmatrix}$.

Thus, one observes that the transfer of a discrete series representation of $\SL(2,\R)$, with highest weight $n$, to $\GL(3,\R)$ has Langlands parameter 
$$z \mapsto \begin{pmatrix}
(\frac{z}{\bar{z}})^\frac{n}{2} & 0 & 0 \\
0 & 1 & 0 \\
0 & 0 & (\frac{z}{\bar{z}})^{-\frac{n}{2}}
\end{pmatrix}; \hskip 20mm j \mapsto \begin{pmatrix}
0 & 0 & 1 \\
0 & -1 & 0 \\
1 & 0 & 0
\end{pmatrix}.$$

We know that this corresponds to a cohomological representation of $\GL(3,\R)$ which is cohomological with respect to the finite dimensional representation with highest weight $(n,0,-n)$. We have already seen that the transfer of $M_n$, the finite dimensional representation of $\SL(2,\R)$ with highest weight $n$, transfers to the finite dimensional representation of $\GL(3,\R)$ with highest weight $(n,0,-n)$. Thus we have the following result:

\begin{prop}
Let $\pi$ be an irreducible unitary cohomological representation of $\SL(2,\R)$ with respect to the finite dimensional representation $M$. Then the representation of $\GL(3,\R)$, $\iota(\pi)$, obtained by the Langlands transfer is cohomological with respect to $\iota(M)$.
\end{prop}

\begin{rmk}
Note that the only cohomological representations of $\SL(2,\R)$ are the discrete series representations and the trivial representation.
\end{rmk}

\bigskip
\section{Vogan-Zuckermann classification for $\Sp(4,\R)$}
\label{V-Z for Sp}
\subsection{$\theta$-stable subalgebras for $\Sp(2n,\R)$}
We will parameterize the $\theta$-stable parabolic subalgebras of $\Sp(2n,\R)$.
We have the following result to aid us in listing all the $\theta$-stable subalgebras of $\Sp(4,\R).$
\begin{lemma}
\label{Clas-theta-parabolics}
The following sets are in $1-1$ correspondence:
\begin{enumerate}
\item \{open, polyhedral root cones in $i\mathfrak{h}/W_K$ \}
\item \{ordered partitions of $n$: $n=\sum\limits_{j=1}^{s}(n_j+m_j)+m$ with $n_j,m_j,m,s \geq 0, n_j+m_j > 0$\}
\end{enumerate}
\end{lemma}
\begin{proof}
Let $x=(x_1,\dots,x_n) \in i\mathfrak{h}/W_K$. Since $W_K$ acts by permuting the coordinates of $x$, we can assume that $x_1 \geq x_2 \geq \dots \geq x_r > 0 > x_{r+1} \geq x_{r+2} \geq \dots \geq x_n$.  This can also be expressed as follows:
$$x=(\underbrace{s,\dots,s}_{n_s},\underbrace{s-1,\dots,s-1}_{n_{s-1}},\dots,\underbrace{1,\dots,1}_{n_1},\underbrace{0,\dots,0}_{m},\underbrace{-1,\dots,-1}_{m_1},\dots,\underbrace{-s,\dots,-s}_{m_s}),$$ with $n=\sum\limits_{j=1}^{s}(n_j+m_j)+m$ with $n_j,m_j,m,s \geq 0, n_j+m_j > 0.$ This gives us a bijection between the two sets above.
\end{proof}

Let $\mathfrak{Q}$ be the set of all $\theta$-stable parabolic subalgebras of $\g.$ The group $K$ acts on the set $\mathfrak{Q}$ via the adjoint action due to which we get a finite set of $\theta$-stable parabolic subalgebras $\mathfrak{Q}/K$. The following lemma gives us a bijection between $\mathfrak{Q}/K$ and open polyhedral root cones in $i\mathfrak{h}/W_K$.

\begin{lemma}
\label{char-theta-sta-subalg}
Every $x \in i\mathfrak{h}/W_K$ defines a $\theta$-stable parabolic subalgebra $\q_x$ by setting \linebreak $\q_x = \lev_x + \uu_x$, where $$\lev_x = \mathfrak{h} \oplus \bigoplus\limits_{\alpha \in \Delta(\g,\mathfrak{h}),\alpha(x)=0}\g_{\alpha}; \hskip 5mm
\uu_x = \bigoplus\limits_{\alpha \in \Delta(\g,\mathfrak{h}),\alpha(x)>0} \g_\alpha.$$
Two $\theta$-stable parabolic subalgebras $\q_x,\q_y$ are equal if and only if $x$ and $y$ are in the same open polyhedral root cone.\\
\noindent
Conversely, up to conjugacy be $K$, any $\theta$-stable parabolic subalgebra $\q$ is $\q=\q_x$ for some $x \in i\mathfrak{h}/W_K$.
\end{lemma}

Two $\theta$-stable parabolic subalgebras, $\q_1,\q_2$ are said to be equivalent if $\Delta(\uu_1 \cap \p) = \Delta(\uu_2 \cap \p)$, i.e. if the non-compact parts in the unipotent radical of the parabolic subalgebras are equal. We list down all the relevant data for $\Sp(4,\C)$.

\subsection{Parabolic subgroups of $G=\Sp(4,\R)$} 
The parabolic subgroup is one of the components in the Langlands inducing data for a representations. To compute the Langlands parameter for $A_{\q}(\lambda)$, we will need to realize $A_{\q}(\lambda)$ as a Langlands quotient of an induced representation. Thus it will be important for us to list down the parabolic subgroups of $\Sp(4,\R)$. The parabolic subgroups containing a Borel subgroup are in bijection with the subsets of a base corresponding to the Borel (see \cite{Spr}).

For $\Sp(4,\R)$, the root system is $\{e_1-e_2,2e_2 \}$. There are $4$ parabolic subgroups of $\Sp(4,\R)$, each corresponding to a subset of the base. One of them is the group itself which corresponds to the full base. This leaves $3$ proper parabolic subgroups of $\Sp(4,\R)$ which are:
\begin{enumerate}
\item \textbf{Minimal parabolic:} The Borel subgroup $B=M_0A_0N_0$, corresponding to the empty subset of the base, with \\ $M_0=\{\begin{pmatrix}
\epsilon_1 & 0 & 0 & 0\\
0 & \epsilon_2 & 0 & 0\\
0 & 0 & \epsilon_1 & 0\\
0 & 0 & 0 & \epsilon_2\\
\end{pmatrix}: \epsilon_i \in \{\pm 1\} \},$\\
$A_0=\{ \text{diag}(a,b,a^{-1},b^{-1}): a,b \in \mathbb{R}_{> 0}^{\times} \},$ and\\
$N_0 = \{ n(x_0,x_1,x_2,x_3) = \begin{pmatrix}
1 & 0 & x_1 & x_2 \\
0 & 1 & x_2 & x_3 \\
0 & 0 & 1 & 0 \\
0 & 0 & 0 & 1 \\
\end{pmatrix} \begin{pmatrix}
1 & x_0 & 0 & 0 \\
0 & 1 & 0 & 0 \\
0 & 0 & 1 & 0 \\
0 & 0 & -x_0 & 1 \\
\end{pmatrix} \} \subset \Sp(4).$
\item \textbf{Siegel parabolic:} The Siegel parabolic $P_S = M_SA_SN_S,$ corresponding to the subset $\Sigma=\{e_1 - e_2\}$ of the base, with\\
$M_S=\{\begin{pmatrix}
m & 0 \\
0 & {}^tm^{-1}
\end{pmatrix}: m \in SL^{\pm}(2,\mathbb{R}) \},$\\
$A_S=\{ \text{diag}(a,a,a^{-1},a^{-1}): a > 0 \},$ and\\
$N_S = \{ \begin{pmatrix}
1_2 & x \\
0 & 1_2
\end{pmatrix}: x={}^tx \in M_2(\R) \}.$

\item \textbf{Jacobi Parabolic:} The Jacobi parabolic $P_J = M_JA_JN_J$, corresponding to the subset $\Sigma=\{2e_2 \}$ of the base, with\\
$M_J=\{\begin{pmatrix}
\epsilon & 0 & 0 & 0\\
0 & a & 0 & b\\
0 & 0 & \epsilon & 0\\
0 & c & 0 & d\\
\end{pmatrix} : \begin{pmatrix}
a & b\\
c & d\\
\end{pmatrix} \in SL(2,\mathbb{R}), \epsilon = \pm 1 \}$\\
$A_J=\{ \text{diag}(a,1,a^{-1},1): a \in \mathbb{R}_{> 0}^{\times} \},$ and\\
$N_J = \{ n(x_0,x_1,x_2,0): x_i \in \R \}.$
\end{enumerate}

We will need these parabolics when we compute the Langlands parameters for the $A_{\q}(\lambda)$'s.

\subsection{$\theta$-stable parabolic subalgebras of $\mathfrak{sp}(4)$}
We list all the $\theta$-stable parabolic subalgebras $\q=\lev + \uu$ of $\mathfrak{sp}(4)$ and the possible admissible characters $\lambda: \lev \rightarrow \C$ which can be obtained from a highest weight of $\mathfrak{h}$. We note that a highest weight of $\mathfrak{h}$ can be extended to an admissible character of $\lev$ if and only if $\lambda\vert_{\mathfrak{h}\cap [\lev_0,\lev_0]}$ and $\lambda \vert_{\mathfrak{a}} = 0$ where the subalgebra $\lev_0=\lev \cap \mathfrak{sp}(4,\R)$ (see \cite{Ha-Ra}). Along with the $\theta$-stable parabolic subalgebras and their corresponding admissible characters, we will also simultaneously list down some useful data for each $\theta$-stable parabolic subalgebra, which will come in handy when we compute the Langlands parameters. To make the list we use Lemma \ref{Clas-theta-parabolics} and Lemma \ref{char-theta-sta-subalg}.

\begin{enumerate}
\label{theta-stable algebras}
\item $x=0$ corresponding to the partition $2=2.$\\
The $\theta$-stable parabolic subalgebra corresponding to $x$ is:
$$\mathfrak{q}_1 = \mathfrak{sp}_4(\mathbb{C}) + 0$$
The Levi part is: $\mathfrak{l} =  \mathfrak{sp}_4(\mathbb{C}).$\\
$\mathfrak{l_0} = \mathfrak{l} \cap \mathfrak{sp}_4(\mathbb{R}) =  \mathfrak{sp}_4(\mathbb{R}).$\\
$[\mathfrak{l_0}, \mathfrak{l_0}] = \mathfrak{l_0}.$\\
So $\mathfrak{h} \cap [\mathfrak{l_0},\mathfrak{l_0}] = \mathfrak{h}.$\\
Therefore,  $\lambda$ of $\mathfrak{h}$ can be extended to get an admissible character of $\mathfrak{l}$ if and only if $\lambda = 0.$ This $\theta$-stable parabolic subalgebra corresponds to the parabolic subgroup $G.$

\item $x=-Z-2Z'$ corresponding to the partition $2=(0+1)+(0+1)+0.$\\
The $\theta$-stable parabolic subalgebra corresponding to $x$ is:
$$\mathfrak{q}_2 = <Z,Z'> + <N_+,X_-,P_{1-},P_{0-}>$$
The Levi part is: $\mathfrak{l} = <Z,Z'>.$\\
$\mathfrak{l_0} = \mathfrak{l} \cap \mathfrak{sp}_4(\mathbb{R}) =  <iZ,iZ'> = \mathfrak{h}.$\\
$[\mathfrak{l_0}, \mathfrak{l_0}] = 0.$\\
So $\mathfrak{h} \cap [\mathfrak{l_0},\mathfrak{l_0}] = 0.$\\
Therefore, any highest weight $\lambda$ of $\mathfrak{h}$ is an admissible character of $\mathfrak{l} = \mathfrak{h}$. This $\theta$-stable parabolic subalgebra corresponds to the parabolic subgroup $B.$

\item $x=2Z-Z'$ corresponding to the partition $2=(1+0)+(0+1)+0.$\\
The $\theta$-stable parabolic subalgebra corresponding to $x$ is:
$$\mathfrak{q}_3 = <Z,Z'> + <N_+,X_+,P_{1+},P_{0-}>$$
The Levi part is: $\mathfrak{l} = <Z,Z'>.$\\
$\mathfrak{l_0} = \mathfrak{l} \cap \mathfrak{sp}_4(\mathbb{R}) =  <iZ,iZ'> = \mathfrak{h}.$\\
$[\mathfrak{l_0}, \mathfrak{l_0}] = 0.$\\
So $\mathfrak{h} \cap [\mathfrak{l_0},\mathfrak{l_0}] = 0.$\\
Therefore, any highest weight $\lambda$ of $\mathfrak{h}$ is an admissible character of $\mathfrak{l} = \mathfrak{h}$. This $\theta$-stable parabolic subalgebra corresponds to the parabolic subgroup $B.$

\item $x=2Z'-Z$ corresponding to the partition $2=(0+1)+(1+0)+0.$\\
The $\theta$-stable parabolic subalgebra corresponding to $x$ is:
$$\mathfrak{q}_4 = <Z,Z'> + <N_-,X_-,P_{1+},P_{0+}>$$
The Levi part is: $\mathfrak{l} = <Z,Z'>.$\\
$\mathfrak{l_0} = \mathfrak{l} \cap \mathfrak{sp}_4(\mathbb{R}) =  <iZ,iZ'> = \mathfrak{h}.$\\
$[\mathfrak{l_0}, \mathfrak{l_0}] = 0.$\\
So $\mathfrak{h} \cap [\mathfrak{l_0},\mathfrak{l_0}] = 0.$\\
Therefore, any highest weight $\lambda$ of $\mathfrak{h}$ is an admissible character of $\mathfrak{l} = \mathfrak{h}$. This $\theta$-stable parabolic subalgebra corresponds to the parabolic subgroup $B.$

\item $x=2Z+Z'$ corresponding to the partition $2=(1+0)+(1+0)+0.$\\
The $\theta$-stable parabolic subalgebra corresponding to $x$ is:
$$\mathfrak{q}_5 = <Z,Z'> + <N_+,X_+,P_{1+},P_{0+}>$$
The Levi part is: $\mathfrak{l} = <Z,Z'>.$\\
$\mathfrak{l_0} = \mathfrak{l} \cap \mathfrak{sp}_4(\mathbb{R}) =  <iZ,iZ'> = \mathfrak{h}.$\\
$[\mathfrak{l_0}, \mathfrak{l_0}] = 0.$\\
So $\mathfrak{h} \cap [\mathfrak{l_0},\mathfrak{l_0}] = 0.$\\
Therefore, any highest weight $\lambda$ of $\mathfrak{h}$ is an admissible character of $\mathfrak{l} = \mathfrak{h}$.
This $\theta$-stable parabolic subalgebra corresponds to the parabolic subgroup $B.$

\item $x=Z+Z'$ corresponding to the partition $2=(2+0)+0.$\\
The $\theta$-stable parabolic subalgebra corresponding to $x$ is:
$$\mathfrak{q}_6 = <Z,Z',N_+,N_-> + <X_+,P_{1+},P_{0+}>$$
Note that $\uu \cap \p = \uu$ which is also equal to the intersection of the unipotent part of $\q_5$ and $\p$. Thus $\q_6$ is equivalent to $\q_5$, and the corresponding $A_\q(\lambda)$'s are isomorphic.

\item $x=-(Z+Z')$ corresponding to the partition $2=(0+2)+0.$\\
The $\theta$-stable parabolic subalgebra corresponding to $x$ is:
$$\mathfrak{q}_7 = <Z,Z',N_+,N_-> + <X_-,P_{1-},P_{0-}>$$
Note that $\uu \cap \p = \uu$ which is also equal to the intersection of the unipotent part of $\q_2$ and $\p$. Thus $\q_7$ is equivalent to $\q_2$, and the corresponding $A_\q(\lambda)$'s are isomorphic.

\item $x=Z$ corresponding to the partition $2=(1+0)+1.$\\
The $\theta$-stable parabolic subalgebra corresponding to $x$ is:
$$\mathfrak{q}_8 = <Z,Z',P_{0+},P_{0-}> + <N_+,X_+,P_{1+}>$$
The Levi part is: $\mathfrak{l} = <Z,Z',P_{0+},P_{0-}>.$\\
$\mathfrak{l_0} = \mathfrak{l} \cap \mathfrak{sp}_4(\mathbb{R}) =  <iZ,iZ',P_{0+}+P_{0-},i(P_{0+}-P_{0-})>.$\\
$[\mathfrak{l_0}, \mathfrak{l_0}] = <\mathfrak{h}_2 = iZ'>.$\\
So $\mathfrak{h} \cap [\mathfrak{l_0},\mathfrak{l_0}] = <\mathfrak{h}_2 = iZ'>.$\\
Therefore, a highest weight $\lambda$ of $\mathfrak{h}$ can be extended to get an admissible character of $\mathfrak{l}$ if and only if $\lambda(\mathfrak{h}_2) = 0.$ Therefore, $\lambda$ has the form $(\lambda_1,0)$. This $\theta$-stable parabolic subalgebra corresponds to the parabolic subgroup $P_J.$

\item $x=-Z'$ corresponding to the partition $2=(0+1)+1.$\\
The $\theta$-stable parabolic subalgebra corresponding to $x$ is:
$$\mathfrak{q}_9 = \langle Z,Z',X_{+},X_{-}\> + \langle N_+,P_{1-},P_{0-}\>$$
The Levi part is: $\mathfrak{l} = \langle Z,Z',X_{+},X_{-}\>.$\\
$\mathfrak{l_0} = \mathfrak{l} \cap \mathfrak{sp}_4(\mathbb{R}) =  \langle iZ,iZ',X_{+}+X_{-},i(X_{+}-X_{-})\>.$\\
$[\mathfrak{l_0}, \mathfrak{l_0}] = \langle iZ, 2i(X_--X_-), 2(X_+-X_-) \>.$\\
So $\mathfrak{h} \cap [\mathfrak{l_0},\mathfrak{l_0}] = \langle \mathfrak{h}_1 = iZ \>.$\\
Therefore, a highest weight $\lambda$ of $\mathfrak{h}$ can be extended to get an admissible character of $\mathfrak{l}$ if and only if $\lambda= (0,\lambda_1).$
Note that this integral weight is conjugate under the Weyl group to an integral weight of the form $\lambda = (\lambda_1,0)$. This $\theta$-stable parabolic subalgebra corresponds to the parabolic subgroup $P_J$.

\item $x=Z-Z'$ corresponding to the partition $2=(1+1)+0.$\\
The $\theta$-stable parabolic subalgebra corresponding to $x$ is:
$$\mathfrak{q}_{10} = \langle Z,Z',P_{1+},P_{1-} \> + \langle N_+,X_+,P_{0-} \>$$
The Levi part is: $\mathfrak{l} = \langle Z,Z',P_{1+},P_{1-} \>.$\\
$\mathfrak{l_0} = \mathfrak{l} \cap \mathfrak{sp}_4(\mathbb{R}) =  \langle iZ,iZ',P_{1+}+P_{1-},i(P_{1+}-P_{1-}) \>.$\\
$[\mathfrak{l_0}, \mathfrak{l_0}] = \langle i(P_{1+}-P_{1-}),-(P_{1+}+P_{1-}), -2i(Z+Z') \>.$\\
So $\mathfrak{h} \cap [\mathfrak{l_0},\mathfrak{l_0}] = \langle \mathfrak{h}_1+\mathfrak{h}_2 \>.$\\
Therefore, a highest weight $\lambda$ of $\mathfrak{h}$ can be extended to get an admissible character of $\mathfrak{l}$ if and only if $\lambda(\mathfrak{h}_1+\mathfrak{h}_2) = 0$ i.e. $\lambda(\mathfrak{h}_1) = -\lambda(\mathfrak{h}_2)$.
Note that such an integral weight is conjugate to an integral weight of the form $(\lambda_1,\lambda_1)$.
This $\theta$-stable parabolic subalgebra corresponds to the parabolic subgroup $P_S$.
\end{enumerate}

We summarize the $\theta$-stable parabolic subalgebras and the relevant data as below:
\begin{center}
\begin{tabular}{|c|c|c|}
\hline
Parabolic & Corresponding & Possible highest\\
subagebras & Parabolic subgroups & weight $\lambda$ \\
\hline
$\q_1$ & $G$ & $\lambda = 0$\\
\hline
$\q_2 \sim \q_7,\q_3,\q_4,\q_5 \sim \q_6$ & $B$ & Any $\lambda$ \\
\hline
$\q_8$ & $P_J$ & $\lambda = (\lambda_1,0)$ \\
\hline
$\q_9$ & $P_J$ & $\lambda = (\lambda_1,0)$ \\
\hline
$\q_{10}$ & $P_S$ & $\lambda = (\lambda_1,\lambda_1)$\\
\hline
\end{tabular}
\end{center}

\section{Parabolic subgroups of $\SO(5,\C)$}
For $G=\Sp(4,\R)$, we know that ${}^LG^\circ=\SO(5,\C)$. Recall that, for a given representation $\pi$ of $G$ the Langlands parameter is a map $\phi(\pi): W_\R \rightarrow {}^LG$ and the image of $W_\R$ under $\phi$ is contained in a parabolic subgroup of ${}^LG^\circ$. Hence, we list down the parabolic subgroups of $\SO(5)$. For $\SO(5,\C)$, the choice of the bilinear form is $J=\text{anti-diag}(1,-1,1,-1,1).$  Then the maximal torus for $\SO(5,\C)$ contains elements of the form $\text{diag}(a,b,1,b^{-1},a^{-1})$. For $\SO(5,\C)$, we have $3$ proper parabolics which are enumerated below:

\begin{enumerate}
\item \textbf{Minimal parabolic:} The Borel $B=M_0A_0N_0$, corresponding to the empty subset of the base, with $M_0=\{I_5\},$ $A_0$ is the subset of the diagonal matrices of the form $A_0=\{ \text{diag}(a,b,1,b^{-1},a^{-1}): a,b \in \mathbb{C}^{\times} \},$ and \\
$N_B = \{ \begin{pmatrix}
1 & * & * & * & *\\
0 & 1 & * & * & *\\
0 & 0 & 1 & * & *\\
0 & 0 & 0 & 1 & *\\
0 & 0 & 0 & 0 & 1\\
\end{pmatrix} \} \subset \SO(5).$

\item \textbf{Siegel parabolic:} The Siegel Parabolic $P_S = M_SA_SN_S$, corresponding to the subset \newline $\Sigma=\{e_1-e_2\}$ of the base, with\\
$M_S=\{\begin{pmatrix}
A & 0 & 0 \\
0 & 1 & 0 \\
0 & 0 & A
\end{pmatrix}: A \in SL(2,\mathbb{C}) \},$\\

$A_S=\{ \text{diag}(a,a,1,a^{-1},a^{-1}): a \in \mathbb{C}^{\times} \},$ and\\

$N_S = \{ \begin{pmatrix}
1 & 0 & * & * & *\\
0 & 1 & * & * & *\\
0 & 0 & 1 & * & *\\
0 & 0 & 0 & 1 & 0\\
0 & 0 & 0 & 0 & 1\\
\end{pmatrix} \} \subset \SO(5).$\\

\item \textbf{Jacobi Parabolic:} The Jacobi parabolic $P_J = M_JA_JN_J$, corresponding to the subset \newline $\Sigma=\{e_2\}$ of the base, with\\
$M_J=\{\begin{pmatrix}
1 & 0 & 0 \\
0 & A & 0\\
0 & 0 & 1\\
\end{pmatrix}: A \in \SO(3) \} \in \SO(5),$\\

$A_J=\{ \text{diag}(a,1,1,1,a^{-1}): a \in \mathbb{C}^{\times} \},$ and\\

$N_J = \{ \begin{pmatrix}
1 & * & * & * & *\\
0 & 1 & 0 & 0 & *\\
0 & 0 & 1 & 0 & *\\
0 & 0 & 0 & 1 & *\\
0 & 0 & 0 & 0 & 1\\
\end{pmatrix} \} \subset \SO(5).$
\end{enumerate}

We can now compute the Langlands parameters for $A_\q(\lambda)$'s and compute their transfers to representations of $\GL(5,\R)$. 

\section[Trivial coefficients]{Cohomological representations with trivial coefficients}
\label{Transfer triv-coeff}
Let $\mathfrak{Q}(\lambda)$ be the set of all non-equivalent $\q$'s such that $\lambda$ can be extended to an admissible character of $\q$. Assuming $\lambda=0$, $\mathfrak{Q}(\lambda)$ consists of all the $8$ nonequivalent $\theta$-stable parabolic subalgebras listed in Section {\ref{theta-stable algebras}}.

\subsection{Trivial and the Discrete Series representations}
For $q_1=\mathfrak{sp}_4(\R)$, the representation $A_\q$ is the trivial representation of $\Sp(4,\R)$. This representation is transferred to the trivial representation of $\GL(5,\R)$ which is cohomological with respect to trivial coefficients.\\

From Remark \ref{DS}, we note that $A_\q$ is the discrete series representations if $\q$ is one of the following:
\begin{itemize}
\item $\q_2 = \langle Z,Z' \> \oplus \langle N_+,X_-,P_{1-},P_{0-} \>,$
\item $\q_3 = \langle Z,Z' \> \oplus \langle N_+,X_+,P_{1+},P_{0-} \>,$
\item $\q_4 = \langle Z,Z' \> \oplus \langle N_-,X_-,P_{1+},P_{0+} \>,$
\item $\q_5 = \langle Z,Z' \> \oplus \langle N_+,X_+,P_{1+},P_{0+} \>.$
\end{itemize}

The transfer of these representations has been dealt with in \cite{Ra-Sa} and we know that the transfer of these representations is cohomological with  respect to the trivial representation of $\GL(5,\R)$.

This leaves us with $3$ $\theta$-stable parabolic subalgebras and their corresponding cohomological representations. The remaining parabolic subalgebras are:
\begin{itemize}
\label{non-temp alg}
\item $\q_8=\langle Z,Z',P_{0+},P_{0-} \> \oplus \langle N_+,X_+,P_{1+} \>$
\item $\q_9=\langle Z,Z',X_{+},X_{-} \> \oplus \langle N_+,P_{1-},P_{0-} \>$
\item $\q_{10}=\langle Z,Z',P_{1+},P_{1-} \> \oplus \langle N_+,X_+,P_{0-} \>$
\end{itemize}

We analyze these case by case. The representations of $\Sp(4,\R)$, which are non-tempered and cohomological with respect to trivial coefficients are listed in Section $2$ of \cite{Oda-Sch}. We include these computations here in some detail.

\subsection{Case of the Jacobi $\theta$-stable subalgebra}
\label{P_J}
We deal with representations of $\Sp(4,\R)$ corresponding to the $\theta$-stable subalgebras $\q_8$ and $\q_9$.

The parabolic subgroup to which $\q_8$ corresponds is the Jacobi parabolic $P_J$. Recall that
$\q_8 = \lev \oplus \mathfrak{u}$, with $\lev = \langle Z,Z',P_{0+}, P_{0-} \rangle$, $\mathfrak{u} = \langle N_+,X_+,P_{1+} \rangle$ and $\lambda = 0$.

We choose a maximally split Cartan subgroup $H$ inside $L$. The Levi $L$ is isomorphic to $\GL(1,\R) \times \SL(2,\R)$. The Lie algebra corresponding to $H=TA$ is $\langle Z, P_{0+}+P_{0-} \rangle$. Note that the Lie algebras of $T$ and $A$ are generated by $Z$ and $P_{0+}+P_{0-} = \begin{pmatrix}
0 & 0 & 0 & 0 \\
0 & 1 & 0 & 0 \\
0 & 0 & 0 & 0 \\
0 & 0 & 0 & -1 \\
\end{pmatrix}$ respectively.

Now let, $\mathrm{Cent}_G(A) = MA.$ Then $M$ is isomorphic to $\SL(2,\R) \times \{\pm 1\}.$ To compute the Langlands parameter for the representation $A_{\q_8}$, we need a parabolic subgroup $P=MAN$ of $G = \Sp(4,\R)$, a discrete series representation on $M$ and a character $\nu$ of $\mathfrak{a}$. For the parabolic, choose any parabolic subgroup of $G$ which has Levi factor $MA$. The Jacobi parabolic $P_J$ is one such subgroup. This corresponds to the subset $\Sigma=\{ 2e_2 \}$ of the base. Thus, the representation $A_{\q_8}$ is obtained as the Langlands quotient of a representation which is induced from the Jacobi parabolic $P_J$.
The character on $\mathfrak{a}$ is obtained by restricting $\rho_L$ to $\mathfrak{a}$. Hence $$\nu = \rho_L\vert_{\mathfrak{a}} = \frac{1}{2}(0,2) = (0,1).$$

Now for the discrete series representation of $M$: The Harish-Chandra parameter for the representation of $M^{\circ} = \SL(2,\R)$, the connected component of $M$, is given by $\rho(u) + \rho(\mathfrak{m} \cap \lev)$ where $\rho$ is computed with respect to $\mathfrak{t}$. Observe that, $M \cap L = \{\pm 1 \}$ which implies that $\rho(\mathfrak{m} \cap \lev) = 0.$ We have $\mathfrak{u} = \langle N_+ , X_+ , P_{1+}\rangle.$ Thus $$\rho(u) = \frac{1}{2}((1+2+1),0) = (2,0).$$ 

The only question remains is whether the representation on $\{\pm 1 \} \subset M$ is the trivial one or the sign character. We compute this as follows:

The Lie algebra of $M$ is $\mathfrak{m} = \langle Z,P_{0+}+P_{0-},X_+,X_- \rangle.$ Then $M \cap K = \{\pm 1 \} \times \SO(2).$
The discrete series representation on  $M \cap K$ is the representation with highest weight given by the formula $2\rho \wedge^{\text{dim } \mathfrak{u} \cap \mathfrak{p}}(\mathfrak{u} \cap \mathfrak{p})\vert_{t}$ which in this case is $(2+1)=3.$ Thus the character on $\{\pm 1\}$ is given by $\epsilon : -1 \mapsto -1.$ Note that this computation gives us the discrete series representation on the $\{\pm 1 \}$ as well as the $\SL(2,\R)$ of the Levi part.

Now we compute the Langlands parameter for $A_{\q_8}.$ Note that since the representation $A_{\q_8}$ is induced from the parabolic $P_J$, the image of $W_\R$ should lie inside the corresponding parabolic subgroup of $P_J \subseteq \SO(5,\C)$.

The transfer of $A_{\q_8}$ to $\GL(5,\R)$ is the Langlands quotient of the following induced representation:
$$\Ind_{P}^G(D_4 \otimes \chi_1 \epsilon \otimes \chi_{-1} \epsilon \otimes \epsilon)$$
where $P$ is the $(2,1,1,1)$ parabolic subgroup of $\GL(5,\R)$ and $\chi_n(x) = x^n,$ since the Langlands parameter for $A_{\q_8}$ is given by
$$z \mapsto \begin{pmatrix}
(z\bar{z}) & 0 & 0 & 0 & 0 \\
0 & (\frac{z}{\bar{z}})^2 & 0 & 0 & 0 \\
0 & 0 & 1 & 0 & 0 \\
0 & 0 & 0 & (z\bar{z})^{-1} & 0 \\
0 & 0 & 0 & 0 & (\frac{z}{\bar{z}})^{-2} 
\end{pmatrix}; \hskip 5mm j \mapsto \begin{pmatrix}
-1 & 0 & 0 & 0 & 0 \\
0 & 0 & 0 & 0 & 1 \\
0 & 0 & -1 & 0 & 0 \\
0 & 0 & 0 & -1 & 0 \\
0 & 1 & 0 & 0 & 0 
\end{pmatrix}.$$

Observe that the Langlands quotient of $\Ind_{P}^G(D_4 \otimes \chi_1 \epsilon \otimes \chi_{-1} \epsilon \otimes \epsilon)$ is isomorphic to \newline $\Ind_P^G(D_4 \otimes \epsilon)$ where $P$ is the $(2,3)$ parabolic of $\GL(5,\R)$, and $\epsilon$ is the sign representation of $\GL(3,\R)$. This follows from the fact that for the Borel $B$ of $\GL(3,\R)$, the Langlands quotient of $\Ind_B^{\GL(3,\R)}(|\cdot| \otimes \epsilon \otimes |\cdot|^{-1})$, is $\epsilon$. We further note that
$\Ind_P^G(D_4 \otimes \epsilon) \cong \Ind_P^G(D_4 \otimes 1) \otimes \epsilon$.
Thus, this is a twist of a unitary representation by the sign character. Hence this representation is unitary. Thus we can appeal to Speh's classification and figure out whether the above representation is cohomological or not.

Since the transferred representation is induced from the $(2,1,1,1)$ parabolic and we only have one factor of $\GL(2,\R)$ in the inducing data, we consider the representation corresponding to the partition $5 = 3 + 2$ of $\GL(5,\R)$ in terms of Speh's classification \cite{Sp2}. 
For the partition $n = 5 = 3 + 2$, we have $n_0 = 3 , n_1 = 2, m_1 = 1$. The representation which is cohomological corresponding to this partition is obtained as a Langlands quotient of the $(2,1,1,1)$ parabolic. The discrete series representation on the $\GL(2,\R)$ part of the Levi is given by $exp(n - \sum\limits_{i = 2}n_i - m_1 )$, which is $4$ since for $i > 1$, $n_i,m_i = 0.$ Thus we observe that the representation which occurs in Speh's classification is $\Ind_P^G(D_4 \otimes 1)$. Thus, the transferred representation obtained from $A_{\q_8}$ does not occur in the classification of Speh. Hence the transfer of $A_{\q_8}$ is not a cohomological representation of $\GL(5,\R)$.

A similar computation for the parabolic subalgebra $\q_9$ shows that the representations $A_{\q_8}$ and $A_{\q_9}$ transfer to the same representation of $\GL(5,\R)$. Thus, the transfer of $A_{\q_8}$ and $A_{\q_9}$ to representations of $\GL(5,\R)$ are not cohomological.

\subsection{Case of Siegel $\theta$-stable parabolic subalgebra}
The last case left is the case when the $\theta$-stable parabolic subalgebra is $\q_{10} = \lev \oplus \mathfrak{u},$ with $\lev = \langle Z,Z',P_{1+},P_{1-} \rangle$ and $\mathfrak{u} = \langle N_+, X_+,P_{0-}  \rangle$.  Let $\lambda = 0$.

We choose a maximally split Cartan subgroup $H$ inside $L$. Then $L$ is isomorphic to $\GL(1,\R) \times \SL(2,\R)$. The Lie algebra corresponding to $H=TA$ is $\langle Z-Z', P_{1+}+P_{1-} \rangle$. Note that the Lie algebras of $T$ and $A$ are generated by $Z-Z'$ and  $P_{1+}+P_{1-} = \begin{pmatrix}
0 & 1 & 0 & 0 \\
1 & 0 & 0 & 0 \\
0 & 0 & 0 & -1 \\
0 & 0 & -1 & 0 \\
\end{pmatrix}$ respectively.

Now let, $\mathrm{Cent}_G(A) = MA.$ Then $M$ is isomorphic to $\SL(2,\R) \times \{\pm 1\}.$
To compute the Langlands parameter for the representation $A_{\q_{10}}$, we need a parabolic subgroup $P=MAN$ of $G = \Sp(4,\R)$, a discrete series representation on $M$ and a character $\nu$ of $\mathfrak{a}$. For the parabolic, choose any parabolic subgroup of $G$ which has Levi factor $MA$. The Siegel parabolic $P_S$ is such a parabolic. This parabolic subgroup corresponds to the subset $\Sigma=\{e_1 - e_2 \}$ of the base. The representation $A_{\q_{10}}$ is obtained as the Langlands quotient of a representation induced from the Siegel parabolic. Now we compute the other two parameters. The character on $\mathfrak{a}$ is obtained by restricting $\rho_L$ to $\mathfrak{a}$. Thus
$$\nu = \rho_L\vert_{\mathfrak{a}} = \frac{1}{2}(2,2) = (1,1).$$

Now for the discrete series representation of $M$: The Harish-Chandra parameter for the representation of $M^{\circ} = \SL(2,\R)$, the connected component of $M$, is given by $\rho(u) + \rho(\mathfrak{m} \cap \lev)$ where $\rho$ is computed with respect to $\mathfrak{t}$. Observe that, $M \cap L = \{\pm 1 \}$ which implies that $\rho(\mathfrak{m} \cap \lev) = 0.$ We have $\mathfrak{u} = \langle N_+ , X_+ , P_{0-}\rangle.$ Thus $$\rho(u) = \frac{1}{2}(2+2+2,2+2+2) = (3,3).$$ 

The only question remains is whether the representation on $\{\pm 1 \} \subset M$ is the trivial one or the sign character. We compute this as follows:

Note that $M \cap K = \{\pm 1 \} \times \SO(2).$
The discrete series representation on  $M \cap K$ is the representation with highest weight given by the formula $2\rho \wedge^{\text{dim } \mathfrak{u} \cap \mathfrak{p}}(\mathfrak{u} \cap \mathfrak{p})\vert_{t}$ which in this case is $(2+2)=4.$ Thus the character on $\{\pm 1\}$ is the trivial character. Now we compute the Langlands parameter for $A_{\q_{10}} \cong \Ind_{P_S}^G(D_3|det|^{\frac{1}{2}}).$

Since the representation $A_{\q_{10}}$ is induced from the parabolic $P_S$, the image of $W_\R$ should go inside $P_S$ which is a parabolic subgroup of $\SO(5,\C)$, corresponding to the subset $\Sigma=\{e_2\}$ of the base.
The Langlands parameter for $A_{\q_{10}}$ is given by
$$z \mapsto \text{diag}(z^2\bar{z}^{-1},z^{-1}\bar{z}^{2},1,z\bar{z}^{-2},z^{-2}\bar{z});$$ and  $$ j \mapsto \begin{pmatrix}
0 & -1 & 0 & 0 & 0 \\
1 & 0 & 0 & 0 & 0 \\
0 & 0 & 1 & 0 & 0 \\
0 & 0 & 0 & 0 & -1 \\
0 & 0 & 0 & 1 & 0 
\end{pmatrix}.$$

Thus the transfer of $A_{\q_{10}}$ to $\GL(5,\R)$ is the Langlands quotient of the following induced representation:
$$\Ind_{P}^G(D_3|det|^{\frac{1}{2}} \otimes 1 \otimes D_3|det|^{-\frac{1}{2}})$$
where $P$ is the $(2,1,2)$ parabolic subgroup of $\GL(5,\R)$. We need to analyze whether this representation occurs in the Speh's classification of unitary irreducible cohomological representations of $\GL(5,\R)$.

We consider the partition $5=1+4$. Using notations from Section \ref{Speh}, we have $n_0=1, n_1=4$ and $m_1=2$. The representation occurring in Speh's classification corresponding to this partition is $\Ind_{(1,4)}^{\GL(5,\R)}(1 \otimes I(3))$.

Now we must compute the corresponding Langlands data for this representation. Appealing to \ref{Speh}, we note that for the character $\chi(3)$ on $T_2^0$ given by $\chi(3)(e^{i\theta_1},e^{i\theta_2})=e^{3i(\theta_1+\theta_2)}$ and $\chi(3)|_{A^2}= exp(\frac{\rho_2}{2})=\frac{a_1}{a_2}$, $$I(3)=J(\chi(3)).$$

But as a representation of $\GL(4,\R)$, $J(\chi(3))= \Ind(D_3|det|^{\frac{1}{2}} \otimes D_3|det|^{-\frac{1}{2}}).$ Thus we note that $\Ind_{(1,4)}^{\GL(5,\R)}(1 \otimes I(3)) = \Ind_{P}^G(D_3|det|^{\frac{1}{2}} \otimes 1 \otimes D_3|det|^{-\frac{1}{2}})$. Hence, the transfer of $A_{\q_{10}}$ occurs in the classification of Speh and is hence cohomological.

\subsection{Summary}
\label{Summary Triv-coeff}
Thus, to summarize we have:
\begin{thm}\label{sp4-main-result-triv-coeff}
Let $\pi$ be an irreducible unitary representation of $\Sp(4,\R)$ such that $\pi$ has non-vanishing cohomology with trivial coefficients. Let $\iota(\pi)$ denote the transferred representation of $\pi$ to $\GL(5,\R)$. Then $\iota(\pi)$ is cohomological with trivial coefficients if $\pi$ is one of the following:
\begin{enumerate}
\item $\pi$ is the trivial representation,
\item $\pi$ is a discrete series representation of $\Sp(4,\R)$,
\item $\pi$ is induced from the Siegel parabolic.
\end{enumerate} 
\end{thm}

\section[Non-Trivial coefficients]{Cohomological representations with Non-Trivial coefficients}
\label{Transfer non-triv coeff}

In this section, we will let $\lambda=(\lambda_1,\lambda_2), \ \lambda_1 \geq \lambda_2 \geq 0$ be a non-zero highest weight of $\Sp(4,\R)$. We split the analysis in the following cases:
\begin{itemize}
\item $\lambda_1 = \lambda_2 \neq 0$,
\item $\lambda_1 > \lambda_2 \neq 0$,
\item $\lambda_2 = 0$. 
\end{itemize}

\subsection{$\lambda = (\lambda_1,\lambda_1)$, with 
$\lambda \neq 0$}

In this case, we note that the $\theta$-stable parabolic subalgebras which are relevant are $\q_2 \sim \q_7,\q_3,\q_4, \q_5 \sim \q_6$ and $\q_{10}$. Note that $A_{\q_i}$ for $2 \leq i \leq 7$ are discrete series representations. Thus, from \cite{Ra-Sa}, we know that these transfer to cohomological representations of $\GL(5,\R)$.

The remaining representations are the representations corresponding to the parabolic $\q_{10}$. As we have already seen, these representations are obtained as the Langlands quotient of a representation which is induced from the Siegel parabolic of $G = \Sp(4,\R)$. We compute the Langlands parameters for the representations $A_{q_{10}}(\lambda)$, as before. We note that the discrete series representation on $M_S$ is given by $2\lambda_1 + 3$. The character on $\mathfrak{a}$ does not change and is still given by $$\nu = \rho_L\vert_{\mathfrak{a}} = \frac{1}{2}(2,2) = (1,1).$$
Thus, the representation $A_{\q_{10}}(\lambda)$ is the irreducible Langlands quotient of the induced representation $\Ind_{P_S}^G(D_{2\lambda_1+3}|det|^{\frac{1}{2}}).$
We note that the Langlands parameter of $A_{\q_{10}}(\lambda)$ is given by
$$z \mapsto \text{diag}(z^{\lambda_1+2}\bar{z}^{-\lambda_1-1},z^{-\lambda_1-1}\bar{z}^{\lambda_1+2},1,z^{\lambda_1+1}\bar{z}^{-\lambda_1-2},z^{-\lambda_1-2}\bar{z}^{\lambda_1+1})$$ and $$ j \mapsto \begin{pmatrix}
0 & (-1)^{2\lambda_1+3} & 0 & 0 & 0 \\
1 & 0 & 0 & 0 & 0 \\
0 & 0 & 1 & 0 & 0 \\
0 & 0 & 0 & 0 & (-1)^{2\lambda_1+3} \\
0 & 0 & 0 & 1 & 0 
\end{pmatrix}.$$

\noindent
Thus, the transfer of $A_{\q_{10}}(\lambda)$ to $\GL(5,\R)$ is obtained as a Langlands quotient of $$\Ind_{(2,1,2)}^{\GL(5,\R)}(D_{2\lambda_1+3}|det|^{\frac{1}{2}}\otimes 1 \otimes D_{2\lambda_1+3}|det|^{-\frac{1}{2}}).$$

The question whether this representation of $\GL(5,\R)$ is cohomological or not does not seem to have an easy answer since the main ingredient, which is the Speh's classification for cohomological representations of $\GL(n,\R)$ with non-trivial coefficients is not available.
The expectation is that this representation is cohomological.

\subsection{$\lambda = (\lambda_1,\lambda_2)$, with $\lambda_2 \neq 0$ and $\lambda_1 > \lambda_2$}

In this case, we note that the $\theta$-stable parabolic subalgebras which are relevant are $\q_2 \sim \q_7,\q_3,\q_4$ and $\q_5 \sim \q_6$. For these subalgebras the Levi parts, $\lev$, are contained in $\kk$ and hence the representations $A_\q$ are the discrete series representations. From \cite{Ra-Sa}, we know that the transfer of these representations are cohomological.

Thus we have:
\begin{prop}
Let $\lambda=(\lambda_1,\lambda_2)$, $\lambda_1 > \lambda_2 \neq 0$. Then the transfer of $A_{\q}(\lambda)$ to $\GL(5,\R)$ is cohomological.
\end{prop}

\subsection{$\lambda = (\lambda_1,0)$}

The $\theta$-stable parabolic subalgebras which are relevant are $\q_2 \sim \q_7,\q_3,\q_4, \q_5 \sim \q_6, \q_8$ and $\q_9$. Out of these $6$ $\theta$-stable parabolic subalgebras, $\q_2 \sim \q_7,\q_3,\q_4$ and $\q_5 \sim \q_6$ correspond to the discrete series representations and we know that these transfer to cohomological representations of $\GL(5,\R)$ from \cite{Ra-Sa}.

This leaves us with the representations $A_{\q_8}(\lambda)$, $A_{\q_9}(\lambda)$. Note that for the $\theta$-stable parabolic subalgebra $\q_8$, $\lambda \vert_\ta = \lambda_1$ and $\lambda \vert_{\mathfrak{a}} = 0.$ These observations along with the calculations in section \ref{P_J} imply that the Langlands parameter for the representation $A_{\q_8}(\lambda)$ is given by:
$$z \mapsto \text{diag}(z\bar{z},(\frac{z}{\bar{z}})^{\frac{\lambda_1+2}{2}},1,(z\bar{z})^{-1},(\frac{z}{\bar{z}})^{-\frac{\lambda_1+2}{2}})$$
$$j \mapsto \begin{pmatrix}
-1 & 0 & 0 & 0 & 0 \\
0 & 0 & 0 & 0 & (-1)^{\lambda_1+2} \\
0 & 0 & 1 & 0 & 0 \\
0 & 0 & 0 & -1 & 0 \\
0 & 1 & 0 & 0 & 0 
\end{pmatrix}.$$

Hence, we observe that the transfer of $A_{\q_8}(\lambda)$ is obtained by taking the Langlands quotient of:
$$\Ind_{P}^G(D_{\lambda_1+2} \otimes \chi_1 \epsilon \otimes \chi_{-1} \epsilon \otimes \epsilon),$$
where $P$ is the $(2,1,1,1)$-parabolic subgroup of $\GL(5,\R)$, $\chi_n(x)=x^n$ and $\epsilon$ is the sign character on $\R^\times$.

A similar calculation as above shows that the transfer of $A_{\q_9}(\lambda)$ is also the Langlands quotient of
$$\Ind_{P}^G(D_{\lambda_1+2} \otimes \chi_1 \epsilon \otimes \chi_{-1} \epsilon \otimes \epsilon),$$ where $P$ is as above.
The question whether this representation of $\GL(5,\R)$ is cohomological or not does not seem to have an easy answer since the main ingredient, which is the Speh's classification for cohomological representations of $\GL(n,\R)$ with non-trivial coefficients is not available at the moment.
The expectation is that this representation is not cohomological. 
\subsection{Summary}
\label{Summary non-triv coeff}
Finally, to summarize the results we put everything in a tabular form. The table completely answers which unitary, irreducible cohomological representations of $G=\Sp(4,\R)$ are transferred to cohomological representations of $\GL(5,\R)$ in the $\lambda=0$ case. In the non-trivial coefficients case, it seems like a difficult question at the moment since no analogous result of Speh's classification for cohomological representations with non-trivial coefficients seem to exist. One hopes to prove this and obtain a complete result for the case $G=\Sp(4,\R).$ 

\begin{center}
\begin{tabular}{|c|c|c|c|}
\hline
Representation & Corresponding & $\lambda$ & Transfer \\ &  Parabolic & & cohomological or not \\
\hline
$A_{\q_1}$ & $B$ & $\lambda = 0$ & Cohomological\\
\hline
$A_{q_2} \cong A_{\q_7},A_{\q_3}$ & $B$ & Any $\lambda$ & Cohomological \\
$A_{\q_4},A_{\q_5} \cong A_{\q_6}$ & & & \\
\hline
$A_{\q_8}$ & $P_J$ & $\lambda=0$ & Not Cohomological \\
 & & $\lambda=(\lambda_1,0)$ & Expected to be not cohomological \\
\hline
$A_{\q_9}$ & $P_J$ & $\lambda=0$ & Not Cohomological \\
 & & $\lambda=(\lambda_1,0)$ & Expected to be not cohomological \\
\hline
$A_{\q_{10}}$ & $P_S$  & $\lambda = 0$ & Cohomological\\
& & $\lambda=(\lambda_1,\lambda_1)$ & Expected to be cohomological\\
\hline
\end{tabular}
\end{center}
\newpage
Considering the observations made above, we make the following conjecture:

\begin{conjecture}
Let $\pi$ be an irreducible unitary representation of $\Sp(4,\R)$ such that $\pi$ has non-vanishing cohomology. Let $\iota(\pi)$ denote the transferred representation of $\pi$ to $\GL(5,\R)$. Then $\iota(\pi)$ is cohomological if $\pi$ is one of the following:
\begin{enumerate}
\item $\pi$ is the trivial representation,
\item $\pi$ is a discrete series representation of $\Sp(4,\R)$,
\item $\pi$ is induced from the Siegel parabolic.
\end{enumerate}
Further, if $\pi$ is cohomological with respect to the finite dimensional representation $M_\lambda$, then $\iota(\pi)$ is cohomological with respect to $\iota(M_\lambda)$.
\end{conjecture}

\bigskip

\end{document}